\documentclass[11pt]{article}
\usepackage{graphicx} 
\usepackage[top=27mm, bottom=27mm, left=25mm, right=25mm]{geometry}
\usepackage{amsfonts}
\usepackage{amssymb}
\usepackage{amsthm}
\usepackage{verbatim}
\usepackage[T1]{fontenc}
\usepackage{graphicx}
\usepackage{diagbox}
\usepackage{xcolor}
\usepackage{amsmath}
\usepackage{array,multirow}
\usepackage{dsfont}
\usepackage{mathcomp}
\usepackage{thmtools}
\usepackage{thm-restate}
\usepackage{enumerate}
\usepackage{bm}
\usepackage{ulem}

\usepackage[pagebackref]{hyperref}
\usepackage[capitalize]{cleveref}
\hypersetup{
	colorlinks,
	linkcolor={blue!60!black},
	citecolor={green!60!black},
	urlcolor={green!60!black}
}

\usepackage{xcolor}
\definecolor{cream}{RGB}{203, 237, 204}

\newtheorem{theorem}{Theorem}[section]
\newtheorem{lemma}[theorem]{Lemma}
\newtheorem{proposition}[theorem]{Proposition}

\newtheorem{definition}[theorem]{Definition}

\newtheorem{remark}[theorem]{Remark}

\newtheorem{observation}[theorem]{Observation}
\newcommand{\F}{{\mathcal{F}}}
\newcommand{\mG}{{\mathcal{G}}}
\newcommand{\A}{{\mathcal{A}}}
\newcommand{\B}{{\mathcal{B}}}
\newcommand{\C}{{\mathcal{C}}}
\newcommand{\mH}{{\mathcal{H}}}
\newcommand{\mF}{{\mathcal{F}}}

\title{Focal-free uniform hypergraphs and codes}

\author{
Xinqi~Huang\thanks{School of Mathematical Sciences, University of Science and Technology of China, Hefei, 230026, Anhui, China. 
    Emails: \texttt{huangxq@mail.ustc.edu.cn},
    \texttt{zhaoyh21@mail.ustc.edu.cn}.}
\and Chong~Shangguan\thanks{Research Center for Mathematics and Interdisciplinary Sciences, Shandong
University, Qingdao, 266237, China; Frontiers Science Center for Nonlinear Expectations, Ministry of
Education, Qingdao, 266237, China. 
    Email: \texttt{theoreming@163.com}.}
\and Xiande~Zhang\thanks{School of Mathematical Sciences, University of Science and Technology of China, Hefei, 230026, Anhui, China; Hefei National Laboratory, University of Science and Technology of China, Hefei, 230088, Anhui, China. 
    Email: \texttt{drzhangx@ustc.edu.cn}.} 
\and Yuhao~Zhao\footnotemark[1]
}
\date{}

\begin{document}
\maketitle
\begin{abstract}
    \noindent Motivated by the study of a variant of sunflowers,  Alon and Holzman recently introduced focal-free hypergraphs. In this paper, we show that there is an interesting connection between the maximum size of focal-free hypergraphs and the renowned Erd\H{o}s Matching Conjecture on the maximum number of edges that can be contained in a uniform hypergraph with bounded matching number. As a consequence, we give asymptotically optimal bounds on the maximum sizes of focal-free uniform hypergraphs and codes, thereby significantly improving the previous results of Alon and Holzman. Moreover, by using the existentce results of combinatorial designs and orthogonal arrays, we are able to explicitly determine the exact sizes of maximum focal-free uniform hypergraphs and codes for a wide range of parameters. 
\end{abstract}
	
\section{Introduction}

Questions asking for the maximum sizes of set systems without containing certain forbidden configurations are widely studied in extremal combinatorics. They usually have various applications in coding theory and theoretical computer science (see, e.g.,~\cite{jukna2011extremal}). In this paper, we study extremal problems concerning forbidding $r$-focal hypergraphs, and present asymptotically optimal bounds on their maximum sizes.

We begin with some definitions. For a positive integer $n$, write $[n]$ for the set $\{1,2,\ldots,n\}$. For a finite set $X$, let $2^X$ denote the power set of $X$ and $\binom{X}{k}$ denote the set of all $k$-subsets of $X$. A hypergraph $\mathcal{F}$ on the vertex set $V(\mathcal{F})$ is a family of distinct subsets (called edges) of $V(\mathcal{F})$. We assume without loss of generality that every vertex in $V(\mathcal{F})$ is contained in at least one edge in $\mathcal{F}$.

\begin{definition}
    A family of $r$ distinct sets $A_0,A_1,\ldots,A_{r-1} \subseteq 2^{[n]}$ is said to be an {\it $r$-focal hypergraph} with focus $A_0$ if for every $x\in A_0$, we have
\begin{align*}
    |\{i\in[r-1]:x\in A_i\}|\ge r-2.
\end{align*}

\noindent A hypergraph $\mathcal{F}\subseteq 2^{[n]}$ is {\it $r$-focal-free} if it does not contain any $r$-focal hypergraph.
\end{definition}

\noindent Note that the above definition is trivial for $r=2$; therefore, we assume throughout that $r\ge 3$. Focal-free hypergraphs were introduced by Alon and Holzman~\cite{alon2023near} when they were studying a variant of sunflowers, called near-sunflowers (see \cite{alon2023near} and the end of this section for more backgrounds). 

A hypergraph $\mathcal{F}\subseteq 2^{[n]}$ is said to be {\it $k$-uniform} if $\mathcal{F}\subseteq\binom{[n]}{k}$. Let $f_r(n,k)$ denote the maximum cardinality of an $n$-vertex $r$-focal-free $k$-uniform hypergraph. Alon and Holzman (see \cite[Theorem 5.2]{alon2023near}) showed that\footnote{In \cite{alon2023near}, uniform focal-free hypergraphs were called {\it ``one-sided focal families''}.} for all $r\ge 3$ and $0\leq k\leq n$, 
\begin{equation}\label{thm-alon-onesided}
    f_r(n,k)\leq (r-1)\cdot\binom{n}{\lceil\frac{(r-2)k}{r-1}\rceil}\bigg/\binom{k}{\lceil\frac{(r-2)k}{r-1}\rceil}.
\end{equation}

\noindent They also commented that ``a lower bound on $f_r(n,k)$ may be obtained by random choice with alterations, but optimizing the bound requires rather messy calculations.''

Quite surprisingly, we notice that packings provide a cheap lower bound for $f_r(n,k)$. For integers $n>k>t\ge 2$, an {\it $(n,k,t)$-packing} is a $k$-uniform hypergraph $\mathcal{P}\subseteq\binom{[n]}{k}$ such that $|A\cap B|<t$ for every two distinct $A,B\in\mathcal{P}$. We observe that every $(n,k,\lceil\frac{(r-2)k}{r-1}\rceil)$-packing $\mathcal{P}$ is $r$-focal-free. 
Indeed, assume otherwise that there exist $A_0,A_1,\ldots,A_{r-1}\subseteq\mathcal{P}$ that form 
an $r$-focal hypergraph with focus $A_0$. Then, $A_0\setminus A_1,\ldots,A_0\setminus A_{r-1}$ must be pairwise disjoint subsets of $A_0$. Therefore, $\sum_{i=1}^{r-1} |A_0\setminus A_i|\le |A_0|=k$, which implies that there exists some $i\in [r-1]$ such that $|A_0\setminus A_i|\le\lfloor\frac{k}{r-1}\rfloor$. Consequently, $|A_0\cap A_i|\ge \lceil\frac{(r-2)k}{r-1}\rceil$, a contradiction. A celebrated result of R\"odl \cite{Rodl-nibble} showed that for fixed $k,t$ and sufficiently large $n$, there exist asymptotically optimal $(n,k,t)$-packings with cardinality $(1-o(1))\cdot\binom{n}{t}/\binom{k}{t}$, where $o(1)\rightarrow 0$ as $n\rightarrow\infty$. Together with the above discussion, it yields that  
\begin{align}\label{eq:cheap-bound}
    (1-o(1))\cdot \binom{n}{\lceil\frac{(r-2)k}{r-1}\rceil}\bigg/\binom{k}{\lceil\frac{(r-2)k}{r-1}\rceil}\le f_r(n,k)\leq (r-1) \cdot \binom{n}{\lceil\frac{(r-2)k}{r-1}\rceil}\bigg/\binom{k}{\lceil\frac{(r-2)k}{r-1}\rceil}.
\end{align}

Note that for $k\ge r-1$ we have $\lceil\frac{(r-2)k}{r-1}\rceil<k$, and determining $f_r(n,k)$ is a {\it degenerate} Tur\'an-type extremal problem for hypergraphs. Conventionally speaking, if there exist two reals $\alpha_r(k)>0$ and $\beta_r(k)\in(0,k)$ such that $\lim_{n\rightarrow\infty}\frac{f_r(n,k)}{n^{\beta_r(k)}}=\alpha_r(k)$, then we call $\alpha_r(k)$ and $\beta_r(k)$ the {\it degenerate Tur\'an density} and {\it Tur\'an exponent} of $f_r(n,k)$, respectively. Clearly, \eqref{eq:cheap-bound} already shows that $\beta_r(k)=\lceil\frac{(r-2)k}{r-1}\rceil$ but it gives no hint on the existence or the exact value of $\alpha_r(k)$. 

Our first main result significantly improves upon \eqref{thm-alon-onesided} (and also on \eqref{eq:cheap-bound}) by explicitly determining the exact values of the degenerate Tur\'an density and Tur\'an exponent of $f_r(n,k)$. In particular, to determine $\alpha_r(k)$, we establish an interesting connection between focal-free uniform hypergraphs and the well-known Erd\H{o}s Matching Conjecture~\cite{erdos1965problem}. Let $m(n,s,\lambda)$ denote the maximum number of edges that can be contained in an $s$-uniform hypergraph on $n$ vertices that does not contain $\lambda$ pairwise disjoint edges. Erd\H{o}s \cite{erdos1965problem} famously conjectured that for all $n\geq s\lambda$,\[m(n,s,\lambda) = \max \left\lbrace  \binom{n}{s}-\binom{n-\lambda+1}{s}, \binom{s\lambda-1}{s} \right\rbrace.\]
The above conjecture has been partially confirmed; see \cite{Frankl-Kupavskii-The-Erods-matching-conjecture} and the references therein for details.

The following theorem shows that $\alpha_r(k)$ can be precisely determined in terms of $m(n,s,\lambda)$.

\begin{theorem}\label{thm:main-hypergraph}
	 For all fixed $r\ge 3$ and $k\geq 2$, we have
	 \begin{equation*}
	 	\lim_{n\to\infty}f_r(n,k)\bigg/\binom{n}{\lceil\frac{(r-2)k}{r-1}\rceil} = 
	 	\frac{1}{\binom{k}{\lceil\frac{(r-2)k}{r-1}\rceil} - m(k,\lfloor\frac{k}{r-1}\rfloor,\lambda)} ,
	 \end{equation*}
	 where $\lambda\in [r-1]$ is the unique integer that satisfies $k + \lambda\equiv 0\pmod{r-1}$.
\end{theorem}

\noindent Note that the upper bound of Theorem~\ref{thm:main-hypergraph}, proved in Theorem~\ref{thm-hypergraphs-upperbound}, improves the upper bound in \eqref{thm-alon-onesided}, since by a result of Frankl \cite{Frankl-upper-bound}, $m(n,s,\lambda) \leq (\lambda-1)\binom{n-1}{s-1}$. 

Moreover, for $\lambda=1$, it is clear that $m(k,\lfloor\frac{k}{r-1}\rfloor,1)=0$; in this case we can, in fact, prove $f_r(n,k)\le\binom{n}{\lceil\frac{(r-2)k}{r-1}\rceil}/\binom{k}{\lceil\frac{(r-2)k}{r-1}\rceil}$. 
An $(n,k,t)$-packing $\mathcal{P}\subseteq\binom{[n]}{k}$ is said to be an {\it $(n,k,t)$-design} if $|\mathcal{P}|=\binom{n}{t}/\binom{k}{t}$. Combining with the known results on the existence of combinatorial designs (see \cite{colbourn2007crc,keevash2014existence}), we are able to determine the exact values of $f_r(n,k)$ for infinitely many parameters.

\begin{proposition}\label{prop:exact-hypergraph}
    Let $r\ge 3$, $r-1 \mid k+1$, and $t=\lceil\frac{(r-2)k}{r-1}\rceil$. Let $n\ge\max\{k,n_0\}$, where $n_0=\binom{k}{k-t} t + t - 1$. If there exists an $(n,k,t)$-design, then $f_r(n,k)=\binom{n}{t}/\binom{k}{t}$. In particular, for every sufficiently large $n$ that satisfies $\binom{k-i}{t-i}\mid \binom{n-i}{t-i}$ for every $0\le i\le t-1$, we have $$f_r(n,k)=\binom{n}{\lceil\frac{(r-2)k}{r-1}\rceil} \bigg/ \binom{k}{\lceil\frac{(r-2)k}{r-1}\rceil}.$$
\end{proposition}

As counterparts of focal-free uniform hypergraphs, focal-free codes are also of interest and were systematically studied by Alon and Holzman~\cite{alon2023near}. 
\begin{definition}\label{def:focal}
    Let $q\geq 2$ be an integer. A family $\bm x^0,\bm x^1,\dots,\bm x^{r-1}$ of $r$ distinct vectors in $[q]^n$ is said to be an {\it $r$-focal code} with focus $\bm x^0$ if for every coordinate $i\in [n]$, we have
\begin{align*}
    |\{j\in[r-1]:x_i^j=x_i^0\}|\ge r-2.
\end{align*}
\noindent A code $\mathcal{C}\subseteq [q]^n$ is {\it $r$-focal-free} if it does not contain an $r$-focal code as a subset.
\end{definition}

Let $f_r^{q}(n)$ denote the maximum cardinality of an $r$-focal-free code in $[q]^{n}$. Alon and Holzman~\cite{alon2023near} proved that 
\begin{equation}\label{thm:alon-main}
    c_r\left( \frac{q}{((q-1)(r-1)+1)^{1/(r-1)}} \right)^n\le f_r^{q}(n) \leq (r-1)q^{\lceil\frac{(r-2)n}{r-1}\rceil},
\end{equation}

\noindent where $c_r>0$ is a positive constant that depends only on $r$. Moreover, for any prime power $q\geq n$, they proved a better lower bound:
\begin{equation}\label{eq:AH-lower-bound}
    f_r^{q}(n) \geq q^{\lceil\frac{(r-2)n}{r-1}\rceil}.
\end{equation}

\noindent Thus, the above results showed that for fixed $r,n$ and sufficiently large $q$, we have 
\begin{equation}\label{eq:focal-free-code-exponent}
    f_r^{q}(n) = \Theta_{n,r}( q^{\lceil\frac{(r-2)n}{r-1}\rceil}).
\end{equation}

\noindent Therefore, it remains a natural problem to determine whether the limit $\lim_{q\to\infty}f_r^{q}(n)/q^{\lceil\frac{(r-2)n}{r-1}\rceil}$ exists. 

Our second main result improves upon \eqref{eq:focal-free-code-exponent} by explicitly determining the value of the above limit.
\begin{theorem}\label{thm:main-code}
	For all fixed $r\geq 3$ and $n\geq 2$, we have
	\[
	\lim\limits_{q\to\infty} f_r^{q}(n) \big/ q^{\lceil\frac{(r-2)n}{r-1}\rceil}
	= \frac{\binom{n}{\lceil\frac{(r-2)n}{r-1}\rceil}}{\binom{n}{\lceil\frac{(r-2)n}{r-1}\rceil}-m(n,\lfloor\frac{n}{r-1}\rfloor,\lambda)},
	\]
	where $\lambda\in [r-1]$ is the unique integer that satisfies $n+\lambda \equiv 0 \pmod{r-1}$.
\end{theorem}

Similarly to \cref{prop:exact-hypergraph}, in the theorem below we obtain some exact results on $f_r^{q}(n)$. Moreover, compared with \cref{prop:exact-hypergraph}, in \cref{prop-exact}, the requirements on the parameters are more flexible. 

\begin{theorem}\label{prop-exact}
	Let $q\geq r-1\geq 2$ and let $q= p_1^{e_1} \cdots p_s^{e_s}$ be the canonical integer factorization of $q\geq 2$, where $p_1,\ldots,p_s$ are distinct primes, $e_1,\ldots,e_s$ are positive integers, and $p_1^{e_1}<\cdots<p_s^{e_s}$. 
	If $r-1\mid n+1$ and $2r-3<n\leq p_1^{e_1}+1$, then \[
	f_r^{q}(n) = q^{\lceil\frac{(r-2)n}{r-1}\rceil}.
	\]
\end{theorem}

We explain the main idea of our proofs here. The upper bounds on the limits in Theorems \ref{thm:main-hypergraph} and \ref{thm:main-code} are essentially proved by double counting (see Theorems \ref{thm-hypergraphs-upperbound} and \ref{thm-codes-upperbound} below). For $q\ge r-1$ and $r-1\mid n+1$, we apply a slightly more sophisticated counting argument and prove a clean upper bound for $f_r^q(n)$ (without the lower order term, see \cref{thm-exact-upper}), which eventually yields the exact result in \cref{prop-exact}. 

The lower bounds on the limits in Theorems \ref{thm:main-hypergraph} and \ref{thm:main-code} are inspired by the aforementioned work of R\"odl \cite{Rodl-nibble}, and its further developments (see \cite{frankl1987colored,Frankl-Rodl-matching,liu2024near,Pippenger-Spencer-asymptotic}). In particular, the lower bound construction of \cref{thm:main-hypergraph} (see \cref{thm-hypergraphs-lowerbound}) is based on a powerful result of Frankl and F\"uredi \cite{frankl1987colored} on the existence of near-optimal induced hypergraph packings; and the lower bound in \cref{thm:main-code} (see \cref{thm-codes-lowerbound}) applies a variant of Frankl and F\"uredi's result, recently obtained by Liu and Shangguan \cite{liu2024near}, to induced packings in multi-partite hypergraphs, where the packings should respect the vertex partition of the host multi-partite hypergraph.   

\paragraph{Outline of the paper.} 
For focal-free uniform hypergraphs, we will prove \cref{thm:main-hypergraph} in \cref{section-hypergraphs}, where the upper and lower bounds of the limit are proved in Sections~\ref{subsection-hypergraphs-upper} and \ref{subsection-hypergraphs-lower}, respectively; furthermore, a short proof of \cref{prop:exact-hypergraph} is also included in \cref{subsection-hypergraphs-upper}. For focal-free codes, since the proof of \cref{thm:main-code} is very similar to that of \cref{thm:main-hypergraph}, we will sketch its proof in \cref{sec:main-code}; however, the proofs of \cref{prop-exact} and \cref{prop:exact-hypergraph} are quite different, and we will prove \cref{prop-exact} in \cref{subsection-codes-exact}. Lastly, we will conclude this paper in Section~\ref{section-conclusion} with some open questions. 

\paragraph{Related work.} We will end this section by mentioning some related work. We remark that Alon and Holzman's study of focal-free hypergraphs and codes was motivated by the study of near-sunflowers, which is a variant of the well-known combinatorial object sunflowers. 

A family of $r$ distinct subsets of $[n]$ is an {\it $r$-near-sunflower} if every $i\in [n]$ belongs to either $0,1,r-1$ or $r$ of the members in this family. 
It is easy to see that if $\mathcal{H}\subseteq 2^{[n]}$ is $r$-near-sunflower-free then it is also $r$-focal-free. Therefore, the $q=2$ case of \eqref{thm:alon-main} shows that every $r$-near-sunflower-free family $\mH \subseteq 2^{[n]}$ has cardinality at most $c^n$, where $c<2$ is a constant depending only on $r$. This in fact proved the Erd\H{o}s--Szemer\'edi-type conjecture for near-sunflowers (which is weaker than the Erd\H{o}s--Szemer\'edi conjecture for sunflowers). Alon and Holzman further posed an Erd\H{o}s--Rado-type conjecture for near-sunflowers (which is again weaker than the Erd\H{o}s--Rado conjecture for sunflowers), and it is still open. For more details on the Erd\H{o}s--Rado and Erd\H{o}s--Szemer\'edi conjectures for sunflowers, see, e.g. \cite{alon2012sunflowers,alweiss2021improved,deuber1997intersection,erdos1960intersection,erdos1978combinatorial,hegedHus2018improved,naslund2017upper}.

As illustrated in \cite{alon2023near}, focal-free hypergraphs and codes are also closely related to cover-free families~\cite{erdHos1985families,frankl1987colored} and frameproof codes~\cite{blackburn2003frameproof,Chee-Zhang-fpc,liu2024near}, which were extensively studied in combinatorics and coding theory. Indeed, when $r=3$, $3$-focal-free hypergraphs are equivalent to $2$-cover-free families, and $3$-focal-free codes are equivalent to $2$-frameproof codes. The interested reader is referred to \cite{alon2023near} for a more detailed discussion on the relation of focal-free  hypergraphs and codes and various other combinatorial objects. 

\section{Focal-free hypergraphs}\label{section-hypergraphs}

The goal of this section is to present the proof of \cref{thm:main-hypergraph}.
We will use the following definition. Let $\mathcal{F}\subseteq 2^{[n]}$ be a hypergraph and $A\in\mathcal{F}$ be an edge. A subset $T\subseteq A$ is called an {\it own subset} of $A$ (with respect to $\F$) if for every $B \in \F\setminus \{A\}$, we have $T\nsubseteq B$; otherwise, $T$ is called a {\it non-own subset} of $A$ (with respect to $\F$).

The following observation presents sufficient and necessary conditions for the existence of an $r$-focal hypergraph. It will be very useful in our proof of \cref{thm:main-hypergraph}.

\begin{observation}\label{obs:r-focal-hypergraph}
    Let $\mathcal{F}$ be a 
    hypergraph with $|\mathcal{F}|\ge r$. Then the following hold:
    \begin{itemize}
        \item [{\rm (i)}] If $A\in\mathcal{F}$ admits a partition $A=T_1\cup\cdots\cup T_{r-1}$ such that for each $i\in[r-1]$, $T_i\neq\emptyset$ and $A\setminus T_i$ is a non-own subset of $A$, 
        then $\mathcal{F}$ contains an $r$-focal hypergraph with focus $A$.

        \item [{\rm (ii)}] If $A,A_1,\ldots,A_{r-1}\in\mathcal{F}$ form an $r$-focal hypergraph with focus $A$, then the $r-1$ members of $\{A\setminus A_i:i\in[r-1]\}$ are pairwise disjoint subsets of $A$.
    \end{itemize}
\end{observation}

\noindent Since the observation follows fairly straightforwardly from the definition of an $r$-focal hypergraph, we omit its proof. For later reference, we remark that \cref{obs:r-focal-hypergraph} (i) and (ii) will be used in the proofs of the upper and lower bounds of the limit in \cref{thm:main-hypergraph}, respectively.

Throughout this section, we will use the following notation. For fixed $r,k$, let $t:=\lceil\frac{(r-2)k}{r-1}\rceil$. 
Then $t=k-\lfloor\frac{k}{r-1}\rfloor$. Moreover, $\lambda\in [r-1]$ satisfies $k+\lambda \equiv 0 \pmod{r-1}$  if and only if 
\begin{align}\label{eq:lambda}
    \lambda(k-t)+(r-1-\lambda)(k-t+1)=k.
\end{align}

\subsection{The upper bound of $f_r(n,k)$}\label{subsection-hypergraphs-upper}

In this subsection, we will prove the following upper bound. 

\begin{theorem}\label{thm-hypergraphs-upperbound}
 

For $r\geq 3$ and $k\geq 2$, let $t=\lceil\frac{(r-2)k}{r-1}\rceil$. When $n\ge\max\{k,n_0\}$, where $n_0=\left(\binom{k}{t} - m(k, k-t, \lambda)\right) t + t - 1$, we have
\begin{equation}\label{eq-hypergraphs-upperbound}
		f_r(n,k)  \leq  \frac{\binom{n}{t}}{\binom{k}{t} - m(k,k-t,\lambda)},
	\end{equation}
	where $\lambda\in [r-1]$ is the unique integer that satisfies $k + \lambda\equiv 0\pmod{r-1}$.
\end{theorem}

The following lemma is needed for the proof of \cref{thm-hypergraphs-upperbound}.
\begin{lemma}\label{lemma:hypergraphs-no own}
    Let $n,k$ and $r$ be integers with $n\ge k\ge 2$ and $r\ge 3$. 
    Let $\mF\subseteq \binom{[n]}{k}$ be an $r$-focal-free hypergraph and $\mF_0=\{A\in\mF:\text{$A$ has no own $(t-1)$-subsets with respect to $\mF$}\}$, where $t=\lceil\frac{(r-2)k}{r-1}\rceil$.
    Then every $A\in \mF_0$ contains at least $\binom{k}{t}- m(k, k-t, \lambda)$ own $t$-subsets with respect to $\mF$.
\end{lemma}

\begin{proof}
    It suffices to show that every $A \in \mF_0$ contains at most $ m(k,k-t,\lambda)$ non-own $t$-subsets.
    Suppose on the contrary that there exists some $A\in \mF_0$ that contains at least $m(k,k-t,\lambda)+1$ non-own $t$-subsets.
    Define
    $$
    \mF_A = 
    \{A\setminus B: \ B \text{ is a non-own } t\text{-subset of } A\}
    \subseteq 
    \binom{A}{k-t}
    .$$
    
    \noindent By the definition of $m(k,k-t,\lambda)$, $\mF_A$ contains $\lambda$ pairwise disjoint members, say $T_1,\ldots,T_{\lambda}\in\binom{A}{k-t}$. By \eqref{eq:lambda}, there exist $r-\lambda -1$ disjoint subsets $T_{\lambda+1},\dots,T_{r-1}\in\binom{A}{k-t+1}$ such that
    $$
    A = T_1\cup T_2\cup \cdots \cup T_{r-1}.
    $$

    \noindent According to the definition of $\mF_A$ and the assumption that $A\in \mF_0$, it is not hard to infer that all of $A\setminus T_i$, $i\in [r-1]$, are non-own subsets of $A$. 
    Therefore, it follows from \cref{obs:r-focal-hypergraph} (i) that $\mathcal{F}$ contains an $r$-focal hypergraph with focus $A$, a contradiction.
\end{proof}

Now we are ready to prove Theorem~\ref{thm-hypergraphs-upperbound}.

\begin{proof}[Proof of Theorem~\ref{thm-hypergraphs-upperbound}]
    Suppose that $\mF \subseteq \binom{[n]}{k}$ is an $r$-focal-free hypergraph. 
    Let $\mF_0$ be defined as in \cref{lemma:hypergraphs-no own} and let
    \[ 
    \mF_1 = \{ A\in \mF : A \text{ contains at least one own $\left(t-1\right)$-subset with respect to } \mF\}.
    \]
    Then $\mF=\mF_0\cup \mF_1$.
    For each $A\in \mF_1$, let
    \[
	  \mathcal{O}_{A}:= \left\{ T\in \binom{A}{t-1} : T \text{ is an own $\left(t-1\right)$-subset of $A$ with respect to $\mF$}\right\},
   \]
   and 
   $$
   \mathcal{B}_A :=\left\{ 
   B\in \binom{[n]}{t}:
   \text{$B$ contains some member in $\mathcal{O}_A$}.
   \right\}
   $$
   Clearly, $A\in \mF_1$ implies that $\mathcal{O}_A$ and $\mathcal{B}_A$ are both nonempty.
   Observe that for distinct $A,A'\in\mathcal{F}_1$, we have $\mathcal{O}_A\cap \mathcal{O}_{A'}=\emptyset$; hence
   $|\cup_{A\in\F_1} \mathcal{O}_A|\ge |\F_1|$.
   Note further that any $t$-set contains $t$ subsets of cardinality $t-1$, and that each $T\in \mathcal{O}_A$ is contained in exactly $n - t +1$ members of $\binom{[n]}{t}$ which belong to $\mathcal{B}_A$.
   By counting the size of the set 
   $$\{ (T, B):  T\in \bigcup\nolimits_{A\in\F_1} \mathcal{O}_A \text{, } B\in \bigcup\nolimits_{A\in\F_1} \mathcal{B}_A 
    \text{ and }
   T\subseteq B
   \}$$ in two ways, one can infer that 
   \begin{equation}\label{eq:onesided-F_0}
       \left|\bigcup\nolimits_{A\in \mF_1} \mathcal{B}_A\right| 
       \ge \left|\bigcup\nolimits_{A\in\F_1} \mathcal{O}_A\right| \cdot \frac{ n - t +1}{t}
       \ge |\mF_1|\cdot \frac{ n - t +1}{t}.
   \end{equation}
   
   For each $M\in \mF_0$, let $\mathcal{C}_M$ be the set consisting of all own $t$-subsets of $M$.
   By Lemma~\ref{lemma:hypergraphs-no own}, we have $|\mathcal{C}_M|\ge \binom{k}{t} - 
   m(k, k-t, \lambda)$.
   By the definition of $\mathcal{C}_M$, it is easy to see that all of the sets $\mathcal{C}_M,~M\in \mF_0$, are pairwise disjoint. Therefore,
   \begin{equation}\label{eq:onesided-F_1}
       \left|\bigcup\nolimits_{M\in \mF_0} \mathcal{C}_M\right| \ge |\mF_0|\cdot\left(\binom{k}{t} - m(k, k-t, \lambda)\right).
   \end{equation}

   \noindent By definition, $\cup_{A\in \mF_1} \mathcal{B}_A$ and $\cup_{M\in \mF_0} \mathcal{C}_M$ are also disjoint. Combining (\ref{eq:onesided-F_0}) and (\ref{eq:onesided-F_1}) yields that
   \begin{align*}
       \binom{n}{t}
       \ge &
       \left|\bigcup\nolimits_{A\in \mF_1} \mathcal{B}_A\right| + \left|\bigcup\nolimits_{M\in \mF_0} \mathcal{C}_M\right| \\
       \ge &
       |\mF_1|\cdot \frac{ n - t +1}{t} + |\mF_0|\cdot \left(\binom{k}{t} - m(k, k-t, \lambda)\right)\\
       \ge & 
       (|\mF_1| + |\mF_0|)\cdot \left( \binom{k}{t} - m(k, k-t, \lambda) \right) \\
       = &
       |\mF|\cdot \left( \binom{k}{t} - m(k, k-t, \lambda) \right)
       ,
   \end{align*}
   as needed, where the last inequality holds when 
   $n\ge 
   \left(\binom{k}{t} - m(k, k-t, \lambda)\right) t + 
   t - 1.
   $
\end{proof}

\begin{remark}\label{remark-hypergraph}
    Using a similar method, one can get rid of the assumption $n\ge (\binom{k}{t} - m(k, k-t, \lambda)) t +  t - 1$, and prove a slightly weaker upper bound with a worse lower order term:
    $$f_r(n,k)\le 
		\frac{\binom{n}{t}}{\binom{k}{t} - m(k,k-t,\lambda)} + \binom{n}{t-1}.$$
    In fact, it is clear that $|\mathcal{F}|=|\mathcal{F}_0|+|\mathcal{F}_1|$. Moreover, it is not hard to deduce from the above proof that 
    $|\mathcal{F}_1|\le\binom{n}{t-1}$ and $|\mathcal{F}_0|\le\frac{
			\binom{n}{t}}{\binom{k}{t} - m(k,k-t,\lambda)
		} $, 
  yielding the desired upper bound.
\end{remark}

Furthermore, when $r-1 \mid k+1$, i.e., $\lambda=1$, the upper bound~\eqref{eq-hypergraphs-upperbound} becomes $\binom{n}{\lceil\frac{(r-2)k}{r-1}\rceil} / \binom{k}{\lceil\frac{(r-2)k}{r-1}\rceil}$. In what follows, we will show that this upper bound can actually be achieved whenever an $(n,k,\lceil\frac{(r-2)k}{r-1}\rceil)$-design exists.

\begin{proof}[Proof of \cref{prop:exact-hypergraph}]
    In Introduction, we have already showed that an $(n,k,\lceil\frac{(r-2)k}{r-1}\rceil)$-packing is also $r$-focal-free. The above discussion and the upper bound~\eqref{eq-hypergraphs-upperbound} for the special case $\lambda=1$ together prove the first half of the proposition.  

    For the second half, just note that Keevash~\cite{keevash2014existence} proved that $(n,k,t)$-designs always exist whenever $n$ is large enough and satisfies $\binom{k-i}{t-i}\mid \binom{n-i}{t-i}$ for all $0\le i\le t-1$ (see also \cite{Delcourt-Postle,Glock-Kuhn-Lo-Osthus} for alternative proofs). 
\end{proof}

\subsection{The lower bound of $f_r(n,k)$}\label{subsection-hypergraphs-lower}

In this subsection, we aim to prove the following lower bound. 

\begin{theorem}\label{thm-hypergraphs-lowerbound}
    For $r\ge 3$ and $k\ge 2$, let $t=\lceil\frac{(r-2)k}{r-1}\rceil$. Then we have

   \begin{equation}
        f_r(n,k) \ge 
        (1-o(1))\cdot\frac{
			\binom{n}{t}}{\binom{k}{t} - m(k,k-t,\lambda)} ,
    \end{equation}
    where $\lambda\in [r-1]$ is the unique integer that satisfies $k + \lambda\equiv 0\pmod{r-1}$, and $o(1)\to 0$ as $n\to \infty$.
\end{theorem}

To prove \cref{thm-hypergraphs-lowerbound}, below we will introduce packings and induced packings of hypergraphs. 
\begin{definition}[Packings and induced packings]\label{def-hypergraph-packing} 
    For a fixed $t$-uniform hypergraph $\mF$ and a {\it host} $t$-uniform hypergraph $\mH$, a family of $m$ $t$-uniform hypergraphs
    $$
    \{
        (V(\mF_1),\mF_1),~(V(\mF_2),\mF_2),\ldots,(V(\mF_m),\mF_m)
    \}
    $$
    forms an $\mF$-packing in $\mH$ if for each $j\in [m]$,
    \begin{itemize}
        \item [{\rm (i)}] $V(\mF_j)\subseteq V(\mH)$,~$\mF_j\subseteq \mH$;
        \item [{\rm (ii)}] $\mF_j$ is a copy of $\mF$ defined on the vertex set $V(\mF_j)$;
        \item [{\rm (iii)}] The $m$ $\mathcal{F}$-copies are pairwise edge disjoint, i.e., $\mF_i \cap \mF_j = \emptyset $ for arbitrary distinct $i,j\in[m]$.
    \end{itemize}
    The above $\mF$-packing is said to be {\it induced} if it further satisfies
    \begin{itemize}
        \item [{\rm (iv)}] $|V(\mF_i)\cap V(\mF_j)|\le t$ for arbitrary distinct $i,j\in [m]$; 
        \item [{\rm (v)}] For $i\neq j$, if $|V(\mF_i)\cap V(\mF_j)|=t$, then $V(\mF_i)\cap V(\mF_j)\notin \mF_i\cup \mF_j$.
    \end{itemize}
\end{definition}

For our purpose, it suffices to use the above definition with $\mathcal{H}=\binom{[n]}{t}$. Since the $t$-uniform hypergraphs in an $\mathcal{F}$-packing are pairwise edge disjoint, it is clear that every $\mathcal{F}$-packing in $\binom{[n]}{k}$ can have at most $\binom{n}{t}/|\mathcal{F}|$ copies of $\mathcal{F}$. The influential works of R\"odl \cite{Rodl-nibble}, Frankl and R\"odl \cite{Frankl-Rodl-matching}, and Pippenger (see \cite{Pippenger-Spencer-asymptotic}) showed that the upper bound is asymptotically tight in the sense that there exists a near-optimal $\mathcal{F}$-packing that contains at least $(1-o(1))\cdot\binom{n}{t}/|\mathcal{F}|$ edge disjoint copies of $\mathcal{F}$. Frankl and F\"uredi \cite{frankl1987colored} strengthened their results by showing that there exists a near-optimal induced $\mathcal{F}$-packing. Their result turns out to be quite useful. It has been used to determine asymptotically the extremal number (or the degenerate Tur\'an density) for several hypergraph extremal problems, see \cite{Alon-Sudakov-1995-disjoint-systems,frankl1987colored,Shangguan-Tamo-degenerate} for a few examples. 

We quote the result of Frankl and F\"uredi as follows.
\begin{lemma}[\cite{frankl1987colored}]\label{lem-frankl}
    Let $k>t$ and $\mathcal{F}\subseteq\binom{[k]}{t}$ be fixed. Then there exists an induced $\mF$-packing
    $
    \{
        (V(\mF_i),\mF_i): i\in[m]
    \}
    $ in $\binom{[n]}{t}$ with $m\ge (1-o(1))\cdot\binom{n}{t}/|\mathcal{F}|$, where $o(1)\to 0$ as $n\to \infty$. 
\end{lemma}

We proceed to present the proof of \cref{thm-hypergraphs-lowerbound}.

\begin{proof}[Proof of Theorem~\ref{thm-hypergraphs-lowerbound}]
    Let $\mG\subseteq\binom{[k]}{k-t}$ be one of the largest $(k-t)$-uniform hypergraphs on $k$ vertices that do not contain $\lambda$ pairwise disjoint edges, where $t=\lceil\frac{(r-2)k}{r-1}\rceil$. By definition, we have $|\mG| = m(k,k-t,\lambda)$. 
     Let $\mG' = \{ [k]\setminus A: A\in\mG\}$ and $\mF = \binom{[k]}{t}\setminus \mG'$.  
    Clearly, we have $|\mG'| = |\mG|$ and $|\mF| = \binom{k}{t} - |\mG'| = \binom{k}{t} - m(k,k-t,\lambda)$.

    Applying \cref{lem-frankl} with $\mF$ defined as above gives an induced $\mF$-packing $\{(V(\mF_i),\mF_i): i\in[m]\}$ in $\binom{[n]}{t}$ with 
    $$m\ge (1-o(1))\cdot\binom{n}{t}/|\mathcal{F}|= (1-o(1))\cdot\frac{\binom{n}{t}}{\binom{k}{t} - m(k,k-t,\lambda)},$$
where $o(1)\to 0$ as $n\to \infty$. Note that for each $i\in [m]$, we have $V(\mF_i)\in\binom{[n]}{k}$. 

    The following observation follows straightforwardly from the construction of $\mF$.
    \begin{observation}\label{obs:property-of-F}
       We have $\mG=\{[k]\setminus A':A'\in\binom{[k]}{t}\setminus\mF\}$. By definition of $\mG$, for every copy $\mF_i$ of $\mF$, the $(k-t)$-uniform hypergraph $\{V(\mF_i)\setminus A':A'\in\binom{V(\mF_i)}{t}\setminus\mF_i\}$ does not contain $\lambda$ pairwise disjoint edges.
    \end{observation}
    
    To prove \cref{thm-hypergraphs-lowerbound}, it suffices to show that $$\mH(\mF):=\{V(\mF_i):i\in[m]\}\subseteq\binom{[n]}{k}$$ is $r$-focal-free. Suppose for the sake of contradiction that $V(\mF_1),\ldots,V(\mF_r)\in\mH(\mF)$ form an $r$-focal hypergraph with focus $V(\mF_r)$. Then, it follows from \cref{obs:r-focal-hypergraph} (ii) that all members of $\{V(\mF_r)\setminus V(\mF_i):i\in[r-1]\}$ are pairwise disjoint subsets of $V(\mF_r)$. Moreover, it follows from the definition of an induced packing (see \cref{def-hypergraph-packing} (iv)) that for each $i\in [r-1]$, $|V(\mF_r)\cap V(\mF_i)|\le t$, which implies that $|V(\mF_r)\setminus V(\mF_i)|\ge k-t$. Combining the discussion above and \eqref{eq:lambda}, it is not hard to verify that there are at least $\lambda$ distinct $i\in [r-1]$ such that $|V(\mF_r)\setminus V(\mF_i)|=k-t$. Assume without loss of generality that $$|V(\mF_r)\setminus V(\mF_1)|=\cdots=|V(\mF_r)\setminus V(\mF_{\lambda})|=k-t.$$ 
    
    \noindent This implies that for each $i\in [\lambda]$, we have $|V(\mF_r)\cap V(\mF_i)|=t$.
    
    It again follows from the definition of an induced packing (see \cref{def-hypergraph-packing} (v)) that for each $i\in [\lambda]$, $V(\mF_r)\cap V(\mF_i)\not\in\mF_r$, and hence $\{V(\mF_r)\cap V(\mF_i):i\in[\lambda]\}\subseteq\binom{V(\mF_r)}{t}\setminus\mF_r$. Consequently, we have $$\{V(\mF_r)\setminus V(\mF_i):i\in[\lambda]\}\subseteq\{V(\mF_r)\setminus A':A'\in\binom{V(\mF_r)}{t}\setminus\mF_r\}$$
    
    \noindent and therefore the latter hypergraph contains $\lambda$ pairwise disjoint edges, which contradicts \cref{obs:property-of-F}. This completes the proof of \cref{thm-hypergraphs-lowerbound}.       
\end{proof} 

\section{Focal-free codes}\label{section-codes}

\subsection{Upper and lower bounds of $f_r^q(n)$}

The proof of \cref{thm:main-code} follows from the same approach as the proof of \cref{thm:main-hypergraph}, which is briefly explained below.

Let us define the own subsequences of a vector, which is an analogy for own subsets of a set. For a vector $\bm x=(x_1,\ldots,x_n)\in[q]^n$ and a subset of indices $T\subseteq[n]$, let $\bm x_T=(x_i:i\in T)$ denote the subsequence of $\bm x$ with coordinates indexed by $T$. For two vectors $\bm x,\bm y\in[q]^n$, let $I(\bm x,\bm y)=\{i\in[n]:x_i=y_i\}$ denote the set of indices of coordinates for which $\bm x$ and $\bm y$ are equal.
    For a code $\mathcal{C}\subseteq [q]^n$ and a codeword $\bm x\in\mathcal{C}$, a subsequence $\bm x_T$ is called an {\it own subsequence} of $\bm x$ (with respect to $\C$) if for every $\bm y \in \C\setminus \{\bm x\}$, $\bm x_T \neq \bm y_T$; otherwise, $\bm x_T$ is called a {\it non-own subsequence} of $\bm x$ (with respect to $\C$).

Similarly to \cref{obs:r-focal-hypergraph}, the following observation presents sufficient and necessary conditions for the existence of an $r$-focal code. 

\begin{observation}\label{obs:r-focal-code}
    Let $\mathcal{C}\subseteq [q]^n$ be a code with $|\mathcal{C}|\ge r$. Then the following hold:
    \begin{itemize}
        \item [{\rm (i)}] If for some $\bm x\in\mathcal{C}$, there is a partition $[n]=T_1\cup\cdots\cup T_{r-1}$ such that for each $i\in[r-1]$, $T_i\neq\emptyset$ and $\bm x_{[n]\setminus T_i}$ is a non-own subsequence of $\bm x$, then $\mathcal{C}$ contains an $r$-focal code with focus $\bm x$.

        \item [{\rm (ii)}] If $\bm x,\bm x^1,\dots,\bm x^{r-1}\in\mathcal{C}$ form an $r$-focal code with focus $\bm x$, then the $r-1$ members of $\{[n]\setminus I(\bm x,\bm x^i):i\in[r-1]\}$ are pairwise disjoint subsets of $[n]$.
    \end{itemize}
\end{observation}

\noindent Since the observation follows directly from the definition of an $r$-focal code, we omit its proof. 

We note that similarly to the usage of \cref{obs:r-focal-hypergraph} in the proof of \cref{thm:main-hypergraph}, \cref{obs:r-focal-code} (i) and (ii) will also be used in the proofs of the upper and lower bounds of the limit in \cref{thm:main-code}, respectively. The main technical difference is that, instead of dealing with own/non-own subsets, we will work with own/non-own subsequences, and will use an appropriate version of \cref{lem-frankl} recently developed in \cite{liu2024near}. We will state our main results below, and postpone their proofs to \cref{sec:main-code}. 

First, the following is an upper bound on $f_r^q(n)$.

\begin{theorem}\label{thm-codes-upperbound}
 For integers $r\geq 3$ and $n\geq 2$, let $t=\lceil\frac{(r-2)n}{r-1} \rceil$. When $q\geq \frac{ t }{n-t+1} \left( \binom{n}{t} - m(n,n-t,\lambda)\right) $, we have
 $$
	f_r^{q}(n) \le\frac{\binom{n}{t}}{\binom{n}{t}-m(n,n-t,\lambda)} q^{t},
	$$
 where $\lambda\in [r-1]$ is the unique integer that satisfies $n+\lambda \equiv 0 \pmod{r-1}$.
\end{theorem}

\begin{remark}
    Similarly to \cref{remark-hypergraph}, one can get rid of the assumption 
    $q\geq \frac{t}{n-t +1} ( \binom{n}{t} - m(n,n-t,\lambda))$, 
    and prove a slightly weaker upper bound with a worse lower order term:
   $$
	f_r^{q}(n) \le 
	\frac{\binom{n}{t}}{\binom{n}{t}-m(n,n-t,\lambda)} q^{t}
	+ \binom{n}{t-1}q^{t-1}.
	$$
    We omit the details.
\end{remark}

We proceed to state a lower bound on $f_r^q(n)$.

\begin{theorem}\label{thm-codes-lowerbound}
    For any integers $r\geq 3$ and $n\geq 2$, let $t=\lceil\frac{(r-2)n}{r-1}\rceil$. Then we have
      $$
    f_r^{q}(n) \geq 
    (1-o(1))\cdot \frac{\binom{n}{t}}{\binom{n}{t}-m(n,n-t,\lambda)}  q^{t},
    $$
    where $\lambda\in [r-1]$ is the unique integer that satisfies $n + \lambda\equiv 0\pmod{r-1}$, and $o(1)\to 0$ as $q\to \infty$.
\end{theorem}

\subsection{Exact values of $f_r^q(n)$}\label{subsection-codes-exact}

In this subsection, we will prove \cref{prop-exact}. First, note that $m(n,\lfloor\frac{n}{r-1}\rfloor,1)=0$ for $\lambda=1$. By \cref{thm-codes-upperbound}, for $n\equiv -1 \pmod{r-1}$ and sufficiently large $q\ge q_0(n)$, where $q_0(n)=2^{\Theta(n)}$, we have $f_r^{q}(n) \le q^{\lceil\frac{(r-2)n}{r-1}\rceil}$. The following theorem, which is the main technical result of this subsection, shows that the same upper bound in fact holds for significantly smaller $q$.

\begin{theorem}\label{thm-exact-upper}
    Let $r\geq 3$. Suppose that $q\geq r-1$ and $r-1\mid n+1$. Then
    \begin{align*}
        f_r^{q}(n) \leq q^{\lceil\frac{(r-2)n}{r-1}\rceil}.
    \end{align*}
\end{theorem}

\noindent Clearly, \cref{thm-exact-upper} implies the upper bound of \cref{prop-exact}. We postpone the proof of \cref{thm-exact-upper} to the end of this subsection. 

Now we turn to the lower bound of \cref{prop-exact}. Indeed, Alon and Holzman \cite{alon2023near} proved exactly the same lower bound under a stronger assumption on the parameters, i.e., $q\ge n$ and $q$ is a prime power (see \eqref{eq:AH-lower-bound}). We will prove our new result by connecting error-correcting codes with large minimum distance to focal-free codes. Note that such a connection was implicitly used in the proof of \eqref{eq:AH-lower-bound} in \cite{alon2023near} (see Proposition 3.2 in \cite{alon2023near}). We will state it explicitly in the next lemma. Note that for any two vectors $\bm x,\bm y \in [q]^n$, the {\it Hamming distance} between $\bm x,\bm y$ is defined by $d(\bm x,\bm y) = |\{i\in [n] : x_i \neq y_i\}|$. The {\it minimum distance} of a code $\C\subseteq [q]^n$, denoted by $d(\C)$, is $\min \{d(\bm x,\bm y) : \text{~distinct~} \bm x,\bm y \in \C\}$. 

\begin{lemma}\label{lem:ECC-to-focal-free}
    If $\C\subseteq [q]^n$ satisfies $d(\C)> \lfloor\frac{n}{r-1}\rfloor$, then $\C$ is $r$-focal-free.
\end{lemma}

\begin{proof}
    Assume otherwise that $\{\bm x^0,\bm x^1,\dots,\bm x^{r-1}\}\subseteq\C$ form an $r$-focal code with focus $\bm x^0$. Then by \cref{obs:r-focal-code} (ii), $\{[n]\setminus I(\bm x^0,\bm x^i):i\in[r-1]\}$ are pairwise disjoint subsets of $[n]$, which implies that $\sum_{j=1}^{r-1} d(\bm x^0,\bm x^j)\le n$. Therefore, there exists some $j\in [r-1]$ such that $d(\bm x^0,\bm x^j) \leq \lfloor\frac{n}{r-1}\rfloor$, a contradiction.
\end{proof}

Alon and Holzman constructed focal-free codes by Reed-Solomon codes~\cite{Reed-Solomon}, which is a classic in coding theory. Here, we will construct focal-free codes by applying \cref{lem:ECC-to-focal-free} to codes generated by orthogonal arrays. More precisely, given positive integers $t,n,q$, an {\it orthogonal array} $OA(t,n,q)$ is an $n\times q^t$ matrix, say $A$, with entries from $[q]$ such that in every $t\times q^t$ submatrix of $A$, every possible vector in $[q]^t$ appears as a column exactly once. Hence, any two different columns of $A$ agree in at most $t-1$ rows. Let $\mathcal{C}\subseteq [q]^n$ be the code formed by the column vectors of $A$. Then, it is easy to see that $d(\mathcal{C})\ge n-t+1$. Consequently, we can construct the codes required in \cref{lem:ECC-to-focal-free} by known results on the existence of orthogonal arrays. 

\begin{lemma}[see {\cite[Theorems \uppercase\expandafter{\romannumeral3}.7.18 and \uppercase\expandafter{\romannumeral3}.7.20, page 226]{colbourn2007crc}}]\label{lem:OA}
    Let $q= p_1^{e_1} \cdots p_s^{e_s}$ be the canonical integer factorization of $q\geq 2$, where $p_1,\ldots,p_s$ are distinct primes, $e_1,\ldots,e_s$ are positive integers, and $p_1^{e_1}<\cdots<p_s^{e_s}$. If $t< p_1^{e_1}$ and $n\leq p_1^{e_1}+1$, then an $OA(t,n,q)$ exists.
\end{lemma}

We prove \cref{prop-exact} by combining \cref{thm-exact-upper} and Lemmas \ref{lem:ECC-to-focal-free} and \ref{lem:OA}.

\begin{proof}[Proof of \cref{prop-exact}]
    The upper bound $f_r^{q}(n) \leq q^{\lceil\frac{(r-2)n}{r-1}\rceil}$ is a direct consequence of \cref{thm-exact-upper}. For the lower bound, according to above discussion and \cref{lem:ECC-to-focal-free}, it is not hard to see that $f_r^{q}(n) \ge q^{\lceil\frac{(r-2)n}{r-1}\rceil}$ as long as there exists an $OA(\lceil\frac{(r-2)n}{r-1}\rceil,n,q)$. In the meantime, the existence of such an orthogonal array follows straightforwardly from \cref{lem:OA}.
\end{proof}

It remains to prove \cref{thm-exact-upper}. 
 
\begin{proof}[Proof of \cref{thm-exact-upper}]
    Write $n=(r-1)d-1$. Then $\lceil\frac{(r-2)n}{r-1}\rceil=n-d+1$. It suffices to show that $f_r^{q}(n)\le q^{n-d+1}$. Suppose for the sake of contradiction that there exists an $r$-focal-free code $\C\subseteq [q]^n$ with $|\C|\ge q^{n-d+1}+1$. Clearly, $d\ge 2$ since $|\C|\le q^n$. For a subset $S\subseteq [n]$, let 
    $$U_S=\{\bm x\in\mathcal{C}:\text{$\bm x_S$ is an own subsequence of $\bm x$ with respect to $\mathcal{C}$}\}.$$
    Then, \cref{obs:r-focal-code} (i) implies that for every partition $[n]=T_1\cup\cdots\cup T_{r-1}$, where $T_i\neq\emptyset$ for each $i\in[r-1]$, we have
    \begin{align}\label{eq:disjoint-union}
        \mathcal{C}=U_{[n]\setminus T_1}\cup\cdots\cup U_{[n]\setminus T_{r-1}}.
    \end{align}
    
    Let $S\in\binom{[n]}{n-d}$ be a subset such that $|U_S|$ is maximized (possibly $|U_S|=0$, and break ties arbitrarily). We will deduce a contradiction by showing that if $q\ge r-1$ and $|\C|\ge q^{n-d+1}+1$, then there must exist some $S'\in\binom{[n]}{n-d}\setminus\{S\}$ such that $|U_{S'}|\ge |U_S|+1$. 
    
    Assume without loss of generality that $S=[n-d]$. We will construct the desired $S'$ by bounding $|U_{[n-d+1]}|$ from the above. Clearly, 
    $$U_{[n-d+1]}=(U_{[n-d+1]}\cap U_{[n-d]})\cup(U_{[n-d+1]}\setminus U_{[n-d]}).$$ It is obvious that $|U_{[n-d+1]}\cap U_{[n-d]}|\le |U_{[n-d]}|$. Moreover, for every $\bm y\in[q]^{n-d}$ (which is not an own subsequence $\bm x_{[n-d]}$ for any $\bm x\in\C$), there are at most $q$ choices of $\bm x\in\C$ such that there exists some $a\in [q]$, for which $(\bm y,a)\in[q]^{n-d+1}$ is an own subsequence $\bm x_{[n-d+1]}$ of $\bm x$. This implies that $$|U_{[n-d+1]}\setminus U_{[n-d]}|\le q\cdot(q^{n-d}-|U_{[n-d]}|).$$ Therefore, we have 
    \begin{align}\label{eq:upper-bound}
        |U_{[n-d+1]}|\le|U_{[n-d]}|+q\cdot(q^{n-d}-|U_{[n-d]}|)=q^{n-d+1}-(q-1)\cdot|U_{[n-d]}|.
    \end{align}

    Observe that $n=(r-2)d+(d-1)$. Consider the partition $[n]=T_1\cup\dots\cup T_{r-1}$, where $T_i=\{(i-1)d+1,\dots,id\}$ for $i\in [r-2]$ and $T_{r-1}=\{(r-2)d+1,\dots,(r-1)d-1\}$. Then $|T_1|=\dots=|T_{r-2}|=d,~|T_{r-1}|=d-1$, and $[n-d+1]=[n]\setminus T_{r-1}$. By \eqref{eq:disjoint-union} and \eqref{eq:upper-bound}, we have
    \begin{align*}
        |U_{[n]\setminus T_1}\cup\cdots\cup U_{[n]\setminus T_{r-2}}|&\ge|\C|-| U_{[n]\setminus T_{r-1}}|=|\C|-|U_{[n-d+1]}|\\
        &\ge(|\C|-q^{n-d+1})+(q-1)\cdot|U_{[n-d]}|\\
        &\ge 1+(q-1)\cdot|U_{[n-d]}|.
    \end{align*}

    \noindent By pigeonhole principle, there exists some $i\in [r-2]$ such that 
    $$|U_{[n]\setminus T_i}|\ge\left\lceil\frac{1+(q-1)\cdot|U_{[n-d]}|}{r-2}\right\rceil\ge |U_{[n-d]}|+1,$$
    where the last inequality holds since $q-1\ge r-2$. Setting $S'=[n]\setminus T_i$, we have $S'\in\binom{[n]}{n-d}$ and $|U_{S'}|>|U_{[n-d]}|$, contradicting the maximality of $U_{[n-d]}$. This completes the proof of the theorem.
\end{proof}

\section{Concluding remarks}\label{section-conclusion}

In this paper, we presented asymptotically tight upper and lower bounds for both focal-free uniform hypergraphs and codes, thus improving the corresponding results of Alon and Holzman~\cite{alon2023near}. In addition, we also determined the exact values of these hypergraphs and codes for infinitely many parameters. Many interesting problems remain open. 

\begin{itemize}
    \item Let $f_r(n)$ denote the maximum cardinality of an $r$-focal-free hypergraph $\F\subseteq 2^{[n]}$. In \cite{alon2023near}, it was observed that the upper and lower bounds on $f_r(n)$ can be proved by $f_r(n)\le\sum_{k=1}^{n} f_r(n,k)$ and by the probabilistic method, respectively. Can we improve these bounds?

    \item We have determined $f_r^q(n)$ asymptotically for fixed $r,n$ and $q\rightarrow\infty$, However, for fixed $r,q$ and $n\rightarrow\infty$, there is still a huge gap between the known upper and lower bounds (see~\cite{alon2023near} for details). So, it would be interesting to narrow this gap. In particular, the upper bound for the binary case is also an upper bound for near-sunflower-free hypergraphs. 

    \item Theorems \ref{thm-hypergraphs-lowerbound} and \ref{thm-codes-lowerbound} are proved by Lemmas \ref{lem-frankl} and \ref{lem-induced-packing}, and therefore the lower bounds are non-explicit. Giving near-optimal explicit constructions for both problems remains open.
\end{itemize}

\section*{Acknowledgements}
\noindent The authors would like to thank Zixiang Xu for telling them that this work was initiated by Xinqi Huang, Xiande Zhang and Yuhao Zhao, and by Chong Shangguan simultaneously and independently, which led to this collaboration. 
The research of Xiande Zhang is supported by the National Key Research and Development Programs of China 2023YFA1010200 and 2020YFA0713100, the NSFC under Grants No. 12171452 and No. 12231014, and the Innovation Program for Quantum Science and Technology 2021ZD0302902. 
The research of Chong Shangguan is supported by the National Natural Science Foundation of China under Grant Nos. 12101364 and 12231014, and the Natural Science Foundation of Shandong Province under Grant No. ZR2021QA005.

{\small\bibliographystyle{abbrv}
\normalem
\bibliography{reference}}

\appendix

\section{Proof of \cref{thm:main-code}}\label{sec:main-code}

In this section, we present proofs of Theorems \ref{thm-codes-upperbound} and \ref{thm-codes-lowerbound}. For fixed $r,n$, let $t:=\lceil\frac{(r-2)n}{r-1}\rceil$. Then $t=n-\lfloor\frac{n}{r-1}\rfloor$. Moreover, $\lambda\in [r-1]$ satisfies $n+\lambda \equiv 0 \pmod{r-1}$ if and only if 
\begin{align}\label{eq:lambda2}
    \lambda(n-t)+(r-1-\lambda)(n-t+1)=n.
\end{align}
For notational convenience, we need the following useful definition that connects subsets of $[q]^n$ to multi-partite hypergraphs.

\begin{definition}\label{def-hypergraph}
    For positive integers $q, n$, let $\mH_n(q)$ denote the complete $n$-partite $n$-uniform hypergraph with equal part size $q$, where the vertex set $V(\mH_n(q))$ admits a partition $V(\mH_n(q)) = V_1\cup \cdots \cup V_n$ with $V_i = \{(i,a): a\in [q]\}$ for each $i\in [n]$, and the edge set is defined as
        $$
 	\mH_n(q) = \{ \{(1,a_1),\ldots,(n,a_n)\}: a_1, a_2, \ldots, a_n\in [q]\}.
 	$$
    Clearly, $|V(\mH_n(q))| = nq$ and $|\mH_n(q)| = q^n$.
 	
    For $t \le n$, let $\mH_n^{(t)}(q)$ denote the complete $n$-partite $t$-uniform hypergraph with equal part size $q$, where    $V(\mH_n^{(t)}(q))=V(\mH_n(q))$, and
 	$$
 	\mH_n^{(t)}(q) = 
 	\{
 	\{(i_1,a_{i_1}),\ldots,(i_t,a_{i_t})\}: 1\le i_1<\cdots< i_t\le n, a_{i_1},\ldots,a_{i_t}\in [q]
 	\}.
 	$$
  Then $|\mH_n^{(t)}(q)|=\binom{n}{t}q^t$.
\end{definition}

Let $\pi:[q]^n\rightarrow\mathcal{H}_n(q)$ be a bijection defined as follows. For each $\bm x\in [q]^n$, let $$\pi(\bm x):=\{(1,x_1),(2,x_2),\dots,(n,x_n)\} \in \mathcal{H}_n(q).$$ For $\mathcal{C}\subseteq [q]^n$, let $\pi(C):=\{\pi(\bm x):\bm x\in\mathcal{C}\}$. Moreover, for $T\subseteq [n]$ and a subsequence $\bm x_T$ of $\bm x$, let $\pi(\bm x_T)=\{(i,x_i):i\in T\}\in \mathcal{H}_n^{(|T|)}(q)$. Crucially, $\pi$ inherits the own subsequence property in the sense that $\bm x_T$ is an own subsequence of $\bm x$ with respect to $\mathcal{C}$ if and only if $\pi(\bm x_T)$ is an own subset of $\pi(\bm x)$ with respect to $\pi(\mathcal{C})$.

\subsection{Proof of \cref{thm-codes-upperbound}}

The following lemma is an analogue to Lemma~\ref{lemma:hypergraphs-no own}.

\begin{lemma}\label{lem:no-own-k-1-implies-many-own-k}
    Let $n\ge 2$ and $r\ge 3$. Suppose that $\C\subseteq [q]^n$ is an $r$-focal-free code. Let $\C_0$ be the set of codewords in $\C$ that have no own $(t-1)$-subsequence with respect to $\C$, where $t=\lceil\frac{(r-2)n}{r-1}\rceil$. Then every $\bm x \in \C_0$ contains at least $\binom{n}{t} - m(n,n-t,\lambda)$ own $t$-subsequences with respect to $\C$.
\end{lemma}

\begin{proof}
    It suffices to show that every $\bm x \in \C_0$ contains at most $ m(n,n-t,\lambda)$ non-own $t$-subsequences. Suppose on the contrary that there exists some $\bm x \in \C_0$ that contains at least $m(n,n-t,\lambda) +1$ non-own $t$-subsequences with respect to $\C$. Let 
    $$\F_{\bm x}:=\{T \in \binom{[n]}{n-t} : \bm x_{[n]\setminus T} \text{ is a non-own $t$-subsequence of } \bm x\}.$$ 
    Then $|\F_{\bm x}|\ge m(n,n-t,\lambda) +1$ and by definition, $\F_{\bm x}$ contains $\lambda$ pairwise disjoint members $T_1,\dots,T_{\lambda}$. By \eqref{eq:lambda2}, there exist  $T_{\lambda+1},\dots,T_{r-1}\in\binom{[n]}{n-t+1}$ such that
    $$[n]=T_1\cup \cdots \cup T_{r-1}.$$
    As $T_1,\dots,T_{\lambda}\in\F_{\bm x}$ and $\bm x \in \C_0$, for each $i\in[r-1]$, $\bm x_{[n]\setminus T_i}$ is a non-own subsequence of $\bm x$. 
    Therefore, it follows from \cref{obs:r-focal-code} (i) that $\mathcal{C}$ contains an $r$-focal code with focus $\bm x$, a contradiction.
\end{proof}

Now we are ready to present the proof of Theorem~\ref{thm-codes-upperbound}.

\begin{proof}[Proof of Theorem~\ref{thm-codes-upperbound}]
    Suppose that $\C\subseteq [q]^n$ is an $r$-focal-free code. Let $\C_0$ be defined as in \cref{lem:no-own-k-1-implies-many-own-k}, and let
	 \[
	 \C_1 = \{\bm x\in \C : \bm x \text{ contains at least one own $\left(t-1\right)$-subsequence with respect to } \C\}.
	 \]
    Clearly, $\C=\C_0 \cup \C_1$. By the discussion above, for each $\bm x\in \C_1$, $\pi(\bm x)$ contains at least one own $(t-1)$-subset with respect to $\pi(\C)$.
      Let
	  \[
	  \mathcal{O}_{\bm x}:= \{ T\in \mH_n^{(t-1)}(q) : T \text{ is an own $\left(t-1\right)$-subset of $\pi(\bm x)$ with respect to $\pi(\C)$}\},
	  \]
	   and 
	  \[
	  \B_{\bm x}:= \{ B\in \mH_n^{(t)}(q) : B \text{ contains some $T\in \mathcal{O}_{\bm x}$}\}.
	  \]
      Clearly, $\bm x\in\C_1$ implies that $\mathcal{O}_{\bm x}$ and $\B_{\bm x}$ are both nonempty. For distinct $\bm x,\bm x'\in\mathcal{C}_1$, we have $\mathcal{O}_{\bm x}\cap \mathcal{O}_{\bm x'}=\emptyset$, so $|\cup_{\bm x\in \C_1} \mathcal{O}_{\bm x}|\ge |\C_1|$. Moreover, every $T=\{(i_1,a_{i_1}),\ldots,(i_{t-1},a_{i_{t-1}})\}\in \mathcal{O}_{\bm x}$ is contained in exactly $(n-t +1)q$ edges in $\mH_n^{(t)}(q)$, say $\{(i_1,a_{i_1}),\ldots,(i_{t-1},a_{i_{t-1}}),(i_t,a_{i_t})\}\in\mH_n^{(t)}(q)$, where $i_t\in[n]\setminus\{i_1,\ldots,i_{t-1}\}$ has $n-t+1$ choices and $a_{i_t}\in[q]$ has $q$ choices. 
      Therefore, by counting the size of the set 
      $$\{ (T, B):  T\in \bigcup\nolimits_{\bm x\in \C_1} \mathcal{O}_{\bm x} \text{, } B\in \bigcup\nolimits_{\bm x\in \C_1} \mathcal{B}_{\bm x}
     \text{ and } T\subseteq B\}$$ in two ways, one can infer that
        \[
	  \left|\bigcup\nolimits_{\bm x\in\C_1} \B_{\bm x}\right| \geq \frac{1}{t} \cdot \left| \bigcup\nolimits_{\bm x\in \C_1} \mathcal{O}_{\bm x} \right| (n-t+1)q
	  \geq |\C_1|\cdot\frac{(n-t +1)q }{t}.
	  \]
	  
	  For each $\bm y \in \C_0$, let
	  \[
	  \mathcal{A}_{\bm y}:= \{ S\in \mH_n^{(t)}(q) : S \text{ is an own $t$-subset of $\pi(\bm y)$ with respect to $\pi(\C)$}\}.
	  \]
	  By Lemma~\ref{lem:no-own-k-1-implies-many-own-k}, $|\mathcal{A}_{\bm y}|\geq \binom{n}{t} - m(n,n-t,\lambda)$ for every $\bm y \in \C_0$.
	  Note that by the definition of own subsequences and subsets, it is routine to check that $\mathcal{A}_{\bm y}$, $\bm y\in\mathcal{C}_0$, are pairwise disjoint; moreover,   $(\cup_{\bm x\in\C_1} \B_{\bm x})\cap(\cup_{\bm y\in\C_0} \A_{\bm y})=\emptyset$. Consequently, 
	  \begin{align*}
	   \binom{n}{t} q^{t}
	     &=  |\mH_n^{(t)}(q)|
	   \geq \left|\bigcup\nolimits_{\bm x\in\C_1} \B_{\bm x}\right| + \left|\bigcup\nolimits_{\bm y\in\C_0} \A_{\bm y}\right| \\
	     &\geq |\C_1|\cdot\frac{(n-t +1)q }{t} +  |\C_0|\cdot\left( \binom{n}{t} - m(n,n-t,\lambda)\right)\\
	     &\geq  (|\C_1|+|\C_0|)\cdot\left( \binom{n}{t} - m(n,n-t,\lambda)\right)\\
	     &=  |\C|\cdot\left( \binom{n}{t} - m(n,n-t,\lambda)\right), 
	\end{align*}
	  where the last inequality holds whenever $q\geq \frac{t}{n-t +1} \left( \binom{n}{t} - m(n,n-t,\lambda)\right)$, as needed.
\end{proof}

\subsection{Proof of \cref{thm-codes-lowerbound}}

The main ingredient for the proof of \cref{thm-codes-lowerbound} is a version of \cref{lem-frankl} stated for multi-partite hypergraphs. We need some more definitions before formally stating it.

Suppose $\mF$ and $\mH$ are $k$-partite hypergraphs with vertex partitions $V(\mF) = \cup_{i = 1}^k W_i$ and $V(\mH) = \cup_{i = 1}^k V_i$, respectively.
A copy $\mF'$ of $\mF$ in $\mH$ is called {\it faithful} (with respect to the partitions above) if for each $i\in [k]$, the copy of $W_i$ in $V(\mF')$ is contained in $V_i$. An $\mF$-packing $\{(V(\mF_i),\mF_i):i\in[m]\}\subseteq\mathcal{H}$ is said to be {\it faithful}, if for every $j\in [m]$, $\mF_j$ is a faithful copy of $\mF$.

\begin{lemma}[\cite{liu2024near}]\label{lem-induced-packing}
    Let $n>t$ and $\mathcal{F}\subseteq\binom{[n]}{t}$ be fixed. Then viewing $\mathcal{F}$ as an $n$-partite hypergraph, there exists a faithful induced $\mF$-packing $\{(V(\mF_i),\mF_i):i\in[m]\}$ in $\mH_n^{(t)}(q)$ with $m\ge(1-o(1))\cdot\frac{\binom{n}{t} q^t}{|\mF|}$, where $o(1)\rightarrow 0$ as $q\rightarrow\infty$.
\end{lemma}

\begin{proof}[Proof of Theorem~\ref{thm-codes-lowerbound}]

 Let $\mG\subseteq\binom{[n]}{n-t}$ be one of the largest $(n-t)$-uniform hypergraphs on $n$ vertices that do not contain $\lambda$ pairwise disjoint edges, where $t=\lceil\frac{(r-2)n}{r-1}\rceil$. Then by definition we have $|\mG|=m(n,n-t,\lambda)$. Let $\mG' = \{
    A\subseteq [n] : [n]\backslash A\in \mG \}$ and $\mF = \binom{[n]}{t}\backslash \mG'$. Clearly, $|\mG|=|\mG'|$ and 
        $|\mF| = \binom{n}{t} - |\mG|
              = \binom{n}{t} - m(n, n-t, \lambda)$.

    Applying \cref{lem-induced-packing} with $\mF$ defined as above gives a faithful induced $\mF$-packing $\{(V(\mF_i),\mF_i): i\in[m]\}$ in $\mH_n^{(t)}(q)$ with
    \begin{align*}
        m\ge (1-o(1))\cdot \frac{\binom{n}{t}q^t}{|\mF|}
        =
        (1- o(1))\cdot\frac{\binom{n}{t}q^t}{\binom{n}{t} - m(n, n-t, \lambda)}.
    \end{align*}
    where $o(1)\to 0$ as $q\to \infty$.
    Note that for each $i\in [m]$, $\mF_i$ is a faithful copy of $\mF$, where both $\mF$ and $\mH_n^{(t)}(q)$ are viewed as $n$-partite hypergraphs. 
    We can treat each copy $V(\mF_i)$ of $V(\mF)$ as a vector in $[q]^n$, according to (the inverse of) $\pi$ defined below \cref{def-hypergraph}. 
    Let $\C=\{\pi^{-1}(V(\mF_i)): i\in[m]\}$.
    
    To prove the theorem, it remains to show that $\C$ is $r$-focal-free. Assume for the sake of contradiction that $\{\bm x^1, \ldots, \bm x^r\}\subseteq\C$ forms an $r$-focal code with focus $\bm x^r$.
    It follows from \cref{obs:r-focal-code} that $[n]\setminus I(\bm x^r,\bm x^1),\ldots,[n]\setminus I(\bm x^r,\bm x^{r-1})$ are $r-1$ pairwise disjoint subsets. Moreover, assume without loss of generality that for each $i\in[r]$, $\bm x^i=\pi^{-1}(V(\mathcal{F}_i))$, where $\mathcal{F}_i$ is a copy of $\mathcal{F}$ in the previous $\mathcal{F}$-packing. Then, for each $i\in[r-1]$, we have $|I(\bm x^r,\bm x^i)|=|V(\mathcal{F}_r)\cap V(\mathcal{F}_i)|\le t$, which implies that $|[n]\setminus I(\bm x^r,\bm x^{i})|\ge n-t$. Combining the above discussion with \eqref{eq:lambda2}, one can infer that there are at least $\lambda$ distinct $i\in[r-1]$ such that $|[n]\setminus I(\bm x^r,\bm x^i)|=n-t$, which implies that $|I(\bm x^r,\bm x^i)|=|V(\mathcal{F}_r)\cap V(\mathcal{F}_i)|=t$. Then we obtain a contradiction due to the same reason as in the proof of \cref{thm-hypergraphs-lowerbound}, we omit the details.
\end{proof}

\end{document}